\documentclass[11pt]{article}
\usepackage{amsmath,amsthm,enumerate,cite,color}
\usepackage{amssymb,graphicx}
\usepackage{authblk}
\usepackage{amsmath}
\usepackage{adjustbox}
\usepackage{rotating}
\usepackage{mathtools}
\usepackage[plain]{fullpage}
\usepackage[text={17cm,24cm}]{geometry}
\usepackage{etoolbox}
\usepackage{cleveref}[2011/12/24]
\providetoggle{long}
\usepackage[none]{hyphenat}
\settoggle{long}{true}
\newtheorem{theorem}{Theorem}[section]
\newtheorem{corollary}[theorem]{Corollary}
\newtheorem{question}{Question}
\newtheorem{lemma}[theorem]{Lemma}

 \title{A characterization of claw-free CIS graphs\\ and new results on the order of CIS graphs}

\author[1]{Liliana Alc\'on\thanks{liliana@mate.unlp.edu.ar}}
\author[1]{Marisa Gutierrez\thanks{marisa@mate.unlp.edu.ar}}
\author[2]{Martin Milani\v{c}\thanks{martin.milanic@upr.si}}

\affil[1]{\normalsize CeMaLP,  Universidad Nacional de La Plata. CONICET, Argentina}
\affil[2]{\normalsize University of Primorska, IAM and FAMNIT, Koper, Slovenia}

\date{\today}

\usepackage{amssymb}
\usepackage{graphicx}
\begin{document}
\maketitle

\begin{abstract}
A graph is CIS if every maximal clique interesects every maximal stable set. Currently, no good characterization or recognition algorithm for the CIS graphs is known. We characterize graphs in which every maximal matching saturates all vertices of degree at least two and use this result to give a structural, efficiently testable characterization of claw-free CIS graphs. We answer in the negative a question of Dobson, Hujdurovi\'c, Milani{\v c}, and Verret [{\it Vertex-transitive CIS graphs}, European J.~Combin.~44 (2015) 87--98] asking whether the number of vertices of every CIS graph is bounded from above by the product of its clique and stability numbers. On the positive side, we show that the question of Dobson et al.~has an affirmative answer in the case of claw-free graphs.
\end{abstract}

\section{Introduction}

Many graph classes can be defined in terms of properties of cliques or stable sets in a graph (see, e.g.,~\cite{MR3575013,HMR2018}).
In this paper we continue the investigation of \emph{CIS graphs}, defined as
graphs in which every maximal clique intersects every maximal stable set. Here, `maximality' refers to maximality under inclusion.
CIS graphs were studied in a series of papers~\cite{MR1344757,AndBorGur,MR3141630,MR2553393,MR2064873,MR3278773,MR3575013,MR2657704,MR2575826}
under different names; the name CIS (Cliques Intersect Stable sets) was suggested by Andrade et
al.~\cite{AndBorGur}. Currently, no good characterization or recognition algorithm for the CIS graphs is known.
 Recognizing CIS graphs was believed to be co-NP-complete~\cite{MR1344757}, conjectured to be co-NP-complete~\cite{MR2234986}, and conjectured
to be polynomial~\cite{AndBorGur}. The difficulty of understanding the structure of CIS graphs is perhaps related to the fact that the class of CIS graphs is not closed under vertex deletion. For example, the \emph{bull}, that is, is the graph with vertex set $\{v_1,\ldots, v_5\}$ and edge set $\{v_1v_2$, $v_2v_3$, $v_3v_4$, $v_2v_5$, $v_3v_5\}$, is a CIS graph, while deleting vertex $v_5$ from the bull yields the $4$-vertex path, which is not CIS.

Some partial results are known regarding the CIS property in particular graph classes. The class of CIS graphs generalizes the class of $P_4$-free graphs, also known as \emph{cographs}~\cite{MR619603,MR3141630}. Polynomially testable characterizations of the CIS property in the classes of planar graphs and of line graphs were given by Sun and Hu~\cite{MR2657704} and by Boros et al.~\cite{MR3575013}, respectively. Vertex-transitive CIS graphs were characterized by Dobson et al.~\cite{MR3278773} and by Hujdurovi\'c~\cite{H2018}. Furthermore, Dobson et al.~proved that vertex-transitive CIS graphs share the well-known property of perfect graphs~\cite{MR0309780} stating that the number of vertices of the graph is bounded from above by the product of its clique number and stability number. They asked whether the property holds for all CIS graphs.

A notion closely related to CIS graphs is that of a strong clique. A clique in a graph $G$ is said to be \emph{strong} if it has non-empty intersection with every maximal stable set of $G$. Thus, a graph is CIS if and only if every maximal clique is strong. A clique is  \emph{simplicial} if it consists of some vertex and all its neighbors. It is not difficult to see that every strong clique is maximal and every simplicial clique is strong.
Hujdurovi\'c et al.~\cite{HMR2018} showed that a clique in a $C_4$-free graph is strong if and only if it is simplicial, which leads to a polynomially testable characterization of CIS $C_4$-free graphs. The concept of strong clique gives rise to several other interesting graph properties studied in the literature (see, e.g.,~\cite{MR3777057,HMR2018,MR3575013,MR3623393}).

\medskip
\noindent{\bf Our results.} Our results consist of two interrelated parts.
First, we give a structural characterization of claw-free CIS graphs, by  proving a composition theorem for this class of graphs (Theorem~\ref{thm:claw-free-CIS}). This leads to a polynomial-time recognition algorithm for the CIS property in the class of claw-free graphs (Corollary~\ref{cor:recognition}). The result is derived using a characterization of graphs in which every maximal matching saturates all vertices of degree at least two (Theorem~\ref{thm:randomly-internally-matchable}), a result related to Sumner's characterization of randomly matchable graphs~\cite{MR530304} that might be of some independent interest.

\begin{sloppypar}
Second, we answer in the negative the question of Dobson et al.~\cite{MR3278773} asking whether the number of vertices of a CIS graph $G$ is necessarily bounded from above by the product of its stability  number, $\alpha(G)$, and clique number, $\omega(G)$. More precisely, using triangle-free graphs of small stability number~\cite{MR1369063}, we construct a sequence of CIS graphs showing that even the relation $|V(G)| = \mathcal{O}(\alpha(G)\omega(G))$ fails for general CIS graphs (Theorem~\ref{thm:counterexamples}). On the positive side, we show that the question of Dobson et al.~has an affirmative answer in the case of claw-free graphs (Theorem~\ref{thm:claw-free-CIS-alpha-omega}).
\end{sloppypar}

\medskip
\noindent{\bf Structure of the paper.} In Section~\ref{sec:preliminaries} we collect the necessary notations and preliminary results. In Section~\ref{sec:random-internally-mathcable} we characterize graphs in which every maximal matching saturates all vertices of degree at least two.
In Section~\ref{sec:claw-free} we prove the structural characterization of claw-free CIS graph. In Section~\ref{sec:bounds} we construct a family of counterexamples to the question of Dobson et al.~and study the question in the case of claw-free graphs. We conclude the paper in Section~\ref{sec:EH} by posing a question left open by our work, namely whether the Erd\H{o}s-Hajnal property holds for the class of CIS graphs.

\section{Preliminaries}\label{sec:preliminaries}

We consider finite, undirected, and non-null graphs only.
Unless specified otherwise by using the term \textit{multigraph},
 all our graphs will be simple, that is, without loops or multiple edges.
  A graph $G = (V,E)$  has vertex set $V(G)=V$ and edge set $E(G)=E$. The \emph{order} of
  $G$ is   $|V|$. Given $S\subseteq V(G)$, the \emph{subgraph induced by $S$} in $G$ is denoted by $G[S]$ and defined as the graph
  with vertex set $S$ and edge set $\{\{x,y\}\mid \{x,y\}\in E; x,y\in S\}$.
 The {\emph complement} ${\overline{G}}$  of  a graph $G = (V,E)$ is the graph with  vertex-set  $V(\overline G)=V $  and the edge-set $E(\overline G) = \{\{x,y\}\mid x,y\in V,~x\neq~y, \textrm{ and } \{x,y\}\not\in E\}$.
  We say that  $G$ is \emph{co-connected} if its complement is connected. A \emph{co-component} of $G$ is the subgraph of $G$ induced by the vertex set of a (connected) component of $\overline{G}$. The \emph{neighborhood} of a vertex $v\in V(G)$ is the set $N_G(v)$
of vertices adjacent to  $v$; its \emph{closed neighborhood} is the set
$N_G[v]$, defined as $N_G[v] = N_G(v)\cup\{v\}$. The cardinality of $N_G(v)$ is the \emph{degree} of $v$,
denoted by $d_G(v)$. A \emph{universal vertex} in a graph $G$ is a vertex of degree $|V(G)|-1$. We denote by
 $\delta(G)$ the minimum degree of a vertex in $G$.  For a set $S\subseteq V(G)$, we let $N_G(S)$ be the set of
  all vertices not in $S$ having a neighbor in $S$.

As usual, we denote the $n$-vertex complete graph, path graph, and cycle graph by $K_n$, $P_n$, and $C_n$, respectively. The graph $K_3$ will be also referred to as the \emph{triangle}. By $K_{m,n}$ we denote the complete bipartite graph with parts of the bipartition of sizes $m$ and $n$. The {\it claw} is the complete bipartite graph $K_{1,3}$. The fact that a graph $G$ is isomorphic to a graph $H$ will be denoted by $G\cong H$. We say that $G$ is \emph{$H$-free} if no induced subgraph of $G$ is isomorphic to $H$.
Furthermore, given a set $\mathcal{F}$ of graphs, we say that a graph $G$ is \emph{$\mathcal{F}$-free} if $G$ is $F$-free for all $F\in \mathcal{F}$.
Given two vertex-disjoint graphs  $G$ and $H$, we denote by $G+H$ their \emph{disjoint union}, that is, the graph with vertex set $V(G)\cup V(H)$ and edge set $E(G)\cup E(H)$. For a non-negative integer $k$, we denote by $kG$ the graph consisting of $k$ disjoint copies of $G$.

A \emph{clique} in a graph is a set of  pairwise adjacent vertices; a
\emph{stable set} (or \emph{independent set}) is a set of pairwise non-adjacent vertices. We say that a clique (resp., stable set) is \emph{maximal} if it is inclusion-maximal,
that is, if it is not contained in any larger clique (resp., stable set).
Given a graph $G$, its \emph{stability number}  (or  \emph{independence number}) is denoted by $\alpha(G)$ and defined
 as the maximum size of a stable set in $G$; furthermore, its
  \emph{clique number} is denoted by $\omega(G)$ and defined as the maximum size of a clique in $G$.

A \emph{matching} in a graph $G$ is a set of pairwise disjoint edges.
Given a matching $M$ and a vertex $v$, we say that $M$ \emph{saturates} $v$ if $M$ contains an edge having $v$ as an endpoint. We will sometimes abuse this terminology and simply say that ``$v$ is in $M$'' if $M$ saturates $v$.
A matching is \emph{perfect} if it saturates all vertices of the graph.
An \emph{internal vertex} in a graph $G$ is a vertex of degree at least two. We say that a matching $M$ in a graph $G$ is a \emph{perfect internal matching} if it saturates all internal vertices of $G$, that is, if every vertex not in $M$ is of degree at most $1$. Perfect internal matchings were studied in a series of papers, see, e.g.,~\cite{Bar_Gombas,Bar_Mik_ARS,Bar_Mik_IPL,MR3488933}.

For undefined graph terminology and notation, we refer to~\cite{MR1367739}.

\subsection{Preliminaries on line graphs of multigraphs}\label{subsec:prelim-line}

Given a multigraph $H$, its \emph{line graph} is the simple graph $L(H)$ with vertex set $E(H)$ in which two distinct vertices $e$ and $e'$ are adjacent if and only if $e$ and $e'$ have a common endpoint as edges in $H$. Clearly, if $G$ is the line graph of a multigraph $H$, then there exists a multigraph $H'$ without loops such that $G$ is (isomorphic to) the line graph of $H'$. Such a multigraph $H'$ can be obtained from $H$ by replacing every loop in $H$ joining $v$ with itself with a pendant edge joining $v$ with a new vertex. Hence, we may assume without loss of generality that all line graphs considered are line graphs of loopless multigraphs.
Given a loopless multigraph $H$ and an edge $e\in E(H)$, let us denote by $w(e)$ the \emph{multiplicity} of $e$ in $H$, that is, the number of edges of $H$ with the same endpoints as $e$. Using this multiplicity function, multigraph $H$ can be equivalently represented with any (simple) graph $\tilde{H}$ with vertex set $V(H)$ obtained from $H$ by keeping only one representative edge from each class of multiple edges, together with the restriction of the multiplicity function $w$ to the edges of $\tilde{H}$.

A \emph{weighted graph} is a pair $(H,w)$ where $H = (V,E)$ is a graph and $w:E\to \mathbb{N}$ is a weight function.\footnote{We denote by $\mathbb{N}$ the set of all (strictly) positive integers.} Interpreting $w$ as the multiplicity function of the edges, we see that every weighted graph $(H,w)$ corresponds to a loopless multigraph. Accordingly, we let  $L(H,w)$ denote
the  \emph{line graph} of $(H,w)$, this is the line graph of the multigraph obtained from $H$ by replacing each edge $e\in E(H)$ with $w(e)$ parallel edges.

\subsection{Preliminaries on CIS graphs}

For $k\ge 2$, a {\it $k$-comb} is a graph $F_k$ with $2k$ vertices
$v_1,\ldots, v_k,w_1,\ldots, w_k$ such that $C = \{v_1,\ldots, v_k\}$ is a clique, $v_i$ is adjacent to $w_i$ for all $i\in \{1,\ldots, k\}$, and there are no other edges. In particular, $S = \{w_1,\ldots, w_k\}$ is a stable set, which shows that $F_k$ is a split graph with a unique split partition $(C,S)$; moreover, $F_k$ is not a CIS graph since $(C,S)$ is a  disjoint pair of a maximal clique and maximal stable set. A {\it $k$-anticomb} is the graph $\overline{F_k}$, the complement of a $k$-comb.

An induced $k$-comb $F_k$ in a graph $G$ is said to be {\it settled} if there exists a vertex $v\in V(G)\setminus V(F_k)$ that is adjacent to every vertex of $C$ and non-adjacent to every vertex of $S$, where $(C,S)$ is the unique split partition of $F_k$. Similarly, an induced $k$-anticomb $\overline{F_k}$ in a graph $G$ with the split partition $(C,S)$ is said to be {\it settled} if there exists a vertex $v\in V(G)\setminus V(\overline{F_k})$ that is adjacent to every vertex of $C$ and non-adjacent to every vertex of $S$.
The following lemma describes a necessary (though in general not sufficient) condition for CIS graphs.

\begin{lemma}[Andrade et al.~\cite{AndBorGur}]\label{lem:CIS-settled}
If $G$ is CIS, then every $k$-comb is settled and every $k$-anticomb is settled.
\end{lemma}

Two vertices $x,y$ in a graph $G$ are said to be \emph{true twins} if $N_G[x] = N_G[y]$. Consider the equivalence \hbox{relation $\sim$} defined on the vertex set of $G$ by the rule $x\sim y$ if and only if $x$ and $y$ are true twins. The \emph{true-twin reduction} of $G$ is the graph obtained from $G$ by contracting each equivalence class of the equivalence relation $\sim$ (which is a clique) into a single vertex. A graph is said to be \emph{true-twin-free} if it coincides with its true-twin reduction.
For later use, we recall the following useful property of CIS graphs (see, e.g.,~\cite{MR3141630,AndBorGur}).

\begin{lemma}\label{lem:components-CIS}
A graph $G$ is CIS if and only if the true-twin reduction of each component of $G$ is CIS.
\end{lemma}

Next, we recall a characterization of CIS line graphs (of simple graphs) due to~Boros et al.~\cite{MR3575013}. The characterization relies on the following concept related to perfect internal matchings. We say that a maximal matching $M$ in a graph $H$ is \emph{absorbing} if every vertex not in $M$ sees at most one edge of $M$, or, more formally, if for every vertex $v\in V(H)$ that is not saturated by $M$, there exists an edge $e$ in $M$ such that every neighbor of $v$ is an endpoint of $e$. (In particular, this implies that $v$ is of degree at most two in $H$.)
Note that if $H$ has an edge, then every maximal matching that is a perfect internal matching is absorbing.

\begin{theorem}[Boros et al.~\cite{MR3575013}]\label{thm:line-CIS}
Let $H$ be a graph without isolated vertices and let $G = L(H)$. Then $G$ is CIS if and only if $H$ has no subgraph isomorphic to a bull and every maximal matching in $H$ is absorbing.
\end{theorem}

\section{Randomly internally matchable graphs}\label{sec:random-internally-mathcable}

A graph $G$ is \emph{randomly matchable} if every matching of $G$ can be extended to a perfect matching, or, equivalently, if every maximal matching of $G$ is perfect. Clearly, a graph $G$ is randomly matchable if and only if each component of $G$ is randomly matchable. Therefore, the following
theorem due to Sumner completely characterizes the randomly matchable graphs.

\begin{theorem}[Sumner~\cite{MR530304}]\label{thm:Sumner}
A connected graph $G$ is randomly matchable if and only if
$G\cong K_{2n}$ or $G\cong K_{n,n}$ for some $n\ge 1$.
\end{theorem}

The concept of perfect internal matchings naturally leads to the following generalization of randomly matchable graphs. We say that a graph $G$ is \emph{randomly internally matchable} if every matching of $G$ can be extended to a perfect internal matching, or, equivalently, if every maximal matching of $G$ is a perfect internal matching. Using this terminology, we note, for later use, the following consequence of Theorem~\ref{thm:line-CIS}.

\begin{corollary}\label{cor:line-triangle-free-CIS-ii}
Let $H$ be a triangle-free graph without isolated vertices and let $G = L(H)$. Then $G$ is CIS if and only if $H$ is randomly internally matchable.
\end{corollary}

\begin{proof}
Immediately from Theorem~\ref{thm:line-CIS}, using the fact that if $H$ is triangle-free, then $H$ has no subgraph isomorphic to a bull and
a maximal matching in $H$ is absorbing if and only if it is a perfect internal matching.
\end{proof}

Clearly, a graph $G$ is randomly internally matchable if and only if each component of $G$ is randomly internally matchable. In the next theorem, we characterize the connected randomly internally matchable graphs. A \emph{leaf} in a graph is a vertex od degree one. A \emph{leaf extension} of a graph $G$ is any graph obtained from $G$ by adding for each vertex $v\in V(G)$ a non-empty set $L_v$ of pairwise non-adjacent new vertices joined to $v$ by an edge.

\begin{theorem}\label{thm:randomly-internally-matchable}
A connected graph $G$ is randomly internally matchable
if and only if $G\cong K_{2n}$ for some $n\ge 1$, $G\cong K_{n,n}$ for some $n\ge 1$, or $G$ is a leaf extension of some graph.
\end{theorem}

\begin{proof}
Sufficiency (the ``if'' direction). If $G\cong K_{2n}$ or $G\cong K_{n,n}$ for some $n\ge 1$, then $G$ is randomly matchable and therefore also randomly internally matchable. Suppose now that $G$ is a leaf extension of a graph $G'$. Let $M$ be a maximal matching in $G$ and let $v\in V(G)$ be an internal vertex of $G$. Then, $v\in V(G')$ and there exists a vertex $v'$ of degree one in $G$ such that $vv'\in E(G)$. If $v$ is not $M$-saturated, then neither is $v'$ and hence $M\cup\{vv'\}$ is a matching properly containing $M$, contrary to the maximality of $M$. Therefore, $v$ is in $M$ and, since $v$ and $M$ were arbitrary, $G$ is randomly internally matchable.

Necessity (the ``only if'' direction). Suppose that $G$ is a connected randomly internally matchable graph. Clearly, $G$ has at least two vertices and hence $\delta(G)\ge 1$. Suppose first that $\delta(G)\ge 2$. Then all vertices of $G$ are internal and hence $G$ is randomly matchable. By Theorem \ref{thm:Sumner}, $G\cong K_{2n}$ or $G\cong K_{n,n}$ for some $n\ge 2$.
Suppose now that $\delta(G)= 1$. We may assume that $G$ has at least three vertices, since otherwise $G\cong K_2$ and we are done.
In particular, for every leaf of $G$, its unique neighbor is an internal vertex. Let $L$ denote the set of all leaves in $G$, let $S= N_G(L)$ denote the set of all neighbors of leaves in $G$, and let $R = V(G)\setminus (L\cup S)$. If $R = \emptyset$, then $G$ is a leaf extension of $G[S]$ and we are done. So we may assume that $R\neq \emptyset$. For every $s\in S$, fix a vertex $s'\in L$ adjacent to $s$ and let $M_S = \{ss'\mid s\in S\}$.
Let $M_R$ be a maximal matching of the graph $G[R]$. Then, $M_R\cup M_S$ is a maximal matching in $G$. Since the set of internal vertices of $G$ is precisely $S\cup R$ and $G$ is randomly internally matchable, $M_R\cup M_S$ saturates all vertices in $S\cup R$. Consequently, $M_R$ saturates all vertices in $R$. Since $M_R$ was an arbitrary maximal matching of $G[R]$, we infer that $G[R]$ is randomly matchable. In particular, every connected component of $G[R]$ has even order. (Theorem \ref{thm:Sumner} exactly characterizes the structure of $G[R]$ but we will not need this characterization in the rest of the proof.)
Since $G$ is connected, it has an edge of the form $rv$ where $r\in R$ and $v\not\in R$. Since no vertex in $R$ has a neighbor in $L$, we have $v\in S$.
Then, the set $M' = (M_S\setminus\{vv'\})\cup \{rv\}$ is a matching in $G$. Extend $M'$ to a maximal matching $M$ in $G$. Since $G$ is randomly internally matchable, $M$ saturates all vertices in $S\cup R$. Let $C$ be the component of $G[R]$ containing $r$ and let $M_C$ be the set of edges of $M$ fully contained in $C$. By construction, every edge in $M$ saturating a vertex in $C\setminus \{r\}$ belongs to $M_C$. However, this implies that $C$ is of odd order, a contradiction with the fact that every component of $G[R]$ has even order. This completes the proof.
\end{proof}

\section{Claw-free CIS graphs}\label{sec:claw-free}

In this section we develop a structural characterization of claw-free CIS graphs. We start with several lemmas giving necessary conditions for claw-free CIS graphs. The {\it gem} is the graph obtained from the $4$-vertex path $P_4$ by adding to it a universal vertex.

\begin{lemma}\label{lem:claw-free-CIS-gem-free}
Let $G$ be a claw-free CIS graph. Then $G$ is gem-free.
\end{lemma}

\begin{proof}
Suppose for a contradiction that $G$ is a claw-free CIS graph containing an induced copy of a gem, say on vertex set $\{s,t,u,v,w\}$ where $(s,t,u,v)$ is a path and $w$ is adjacent to all vertices in $\{s,t,u,v\}$.
Since the subgraph of $G$ induced by $\{s,t,u,v\}$ is isomorphic to an induced $2$-comb and $G$ is CIS, Lemma \ref{lem:CIS-settled} implies that there exists a vertex $x\in V(G)\setminus \{s,t,u,v\}$ such that $\{tx,ux\}\subseteq E(G)$
and $\{sx,vx\}\cap E(G) = \emptyset$.
Clearly $x\neq w$.
Furthermore, $xw\not\in E(G)$, since otherwise $\{w,s,v,x\}$ would induce a claw in $G$. This implies that $\{s,t,u,v,w,x\}$ induces a $3$-anticomb in $G$. (Indeed, the complement of $G$ contains a comb $F_3$ having a clique $\{s,v,x\}$ and a stable set $\{t,u,w\}$.) By Lemma \ref{lem:CIS-settled}, $G$ contains a vertex $y\in V(G)\setminus \{s,t,u,v,w,x\}$ adjacent to every vertex in the clique $\{t,u,w\}$ and non-adjacent to every vertex in the stable set $\{s,v,x\}$. But now, the vertex set $\{s,t,x,y\}$ induces a claw in $G$, a contradiction.
\end{proof}

We denote by $W_4$ the {\it $4$-wheel}, that is,  the graph obtained from the $4$-vertex cycle $C_4$ by adding to it a universal vertex.

\begin{lemma}\label{lem:claw-free-gem-free}
Every connected and co-connected $\{$claw, gem$\}$-free graph is $W_4$-free.
\end{lemma}

\begin{proof}
Suppose for a contradiction that $G$ is a connected and co-connected
 claw-free gem-free graph that is not $W_4$-free. Then, $G$ contains an induced $4$-cycle $(x_1,x_2,x_3,x_4)$
  and a vertex $z$ adjacent to all vertices in $X = \{x_1,x_2,x_3,x_4\}$. We claim that every vertex
  of $G$ not in $X$ has a neighbor in $X$. Suppose that this is not the case. Since $G$ is connected,
  we may assume that $G$ contains an induced path $(y_1,y_2,x_1)$ such that $y_1$ has no neighbors in $X$.
  Note that $y_2$ is not adjacent to $x_3$, since otherwise $G$ would contain an induced claw on vertex set $\{y_2,y_1,x_1,x_3\}$.
Furthermore, $y_2$ is adjacent to exactly one of $x_2$ and $x_4$ since otherwise $G$ would contain an
induced claw on vertex set
$\{y_2,y_1,x_2,x_4\}$ (if $y_2$ is adjacent to both $x_2$ and $x_4$) or on vertex set $\{y_1,x_1,x_2,x_4\}$ (if $y_2$ is
adjacent to neither $x_2$ nor $x_4$).
 Hence, without loss of generality, we may assume that $y_2$ is adjacent to $x_2$ but not to $x_4$.
Consequently, $G$ contains an induced gem either on vertex set
$\{y_2,x_2,x_3,x_4,z\}$ (if $y_2$ is adjacent to $z$) or on vertex set $\{y_2,x_1,z,x_3,x_2\}$ (otherwise).
This contradiction shows that every vertex not in $X$ has a neighbor in $X$. Even more, since $G$ is claw-free, every vertex not in $X$ has at least two neighbors in $X$.

Since $G$ is co-connected, its complement $\overline{G}$ is connected. Let $Z$ be the set of vertices of
$\overline{G}$ with no neighbors in $X$.
Note that $Z$ is non-empty since $z\in Z$. Since $\overline{G}$ is connected,
we may assume that $\overline{G}$ contains an induced path $(z_1,z_2,x_1)$ with $z_1\in Z$.
This means that in $G$ there exist non-adjacent vertices $z_1$ and $z_2$
 such that $z_2$ is
non-adjacent to $x_1$, and $z_1$ is adjacent to all vertices of $X$.
Since we have proved that every vertex of $G$ not in $X$ has at least two neighbors
 in $X$, without loss of generality, we can assume that $z_2$ is adjacent to
$x_2$. And therefore, $z_2$ must be adjacent also to $x_3$,
otherwise the vertices $\{z_2,x_1,x_2,x_3\}$ induce a claw in  $G$.
 Since $z_1$ and $z_2$ are non-adjacent in $G$, we infer that $G$ contains an induced gem on vertex set $\{x_1,z_1,x_3,z_2,x_2\}$, a contradiction.
\end{proof}

Kloks et al.~showed in~\cite{MR1354190} that the class of $\{$claw, gem, $W_4\}$-free graphs is exactly the class of {\it dominoes}, that is, graphs in which each vertex is contained in at most two maximal cliques. Furthermore, the dominoes are precisely the line graphs of triangle-free multigraphs; see also~\cite{MR2002172}. This result together with
Lemmas~\ref{lem:claw-free-CIS-gem-free} and~\ref{lem:claw-free-gem-free} implies the following.

\begin{corollary}\label{cor:claw-free-CIS-line}
Every connected and co-connected claw-free CIS graph is the line graph of a connected triangle-free multigraph.
\end{corollary}

A lemma similar to Lemma~\ref{lem:components-CIS} holds for claw-free graphs. The lemma follows immediately from the definitions.

\begin{lemma}\label{lem:components-claw-free}
A graph $G$ is claw-free if and only if the true-twin reduction of each component of $G$ is claw-free.
\end{lemma}

\begin{sloppypar}
Lemmas~\ref{lem:components-CIS} and~\ref{lem:components-claw-free} imply that when studying CIS claw-free graphs, we may restrict our attention to connected true-twin-free graphs. Thus, the following theorem gives a complete structural characterization of claw-free CIS graphs. Given a graph $G$, the \emph{corona} of $G$ (with $K_1)$ is the graph $G\circ K_1$ obtained from $G$ by adding for each vertex $v\in V(G)$ a new vertex $v'$ and making $v'$ adjacent to $v$. Note that the corona of $G$ is a particular leaf extension of it.
\end{sloppypar}

\begin{theorem}\label{thm:claw-free-CIS}
Let $G$ be a connected true-twin-free claw-free graph.
Then $G$ is CIS if and only if one the following holds:
\begin{enumerate}
\item $G\cong \overline{pK_2+qK_1}$ for some $p\ge 0$ and $q\in \{0,1\}$.
\item $G\cong L(K_{n,n})$ for some $n\ge 1$.
\item $G\cong L(G'\circ K_1)$ for some triangle-free graph $G'$.
\end{enumerate}
\end{theorem}

\begin{proof}
\noindent{Sufficiency} (the ``if'' direction).
Suppose first that $G\cong \overline{pK_2+qK_1}$ for some $p\ge 0$ and $q\in \{0,1\}$. Since the class of CIS graphs is closed under complementation,
it suffices to show that the graph $pK_2+qK_1$ is CIS. This follows easily from the definition and Lemma \ref{lem:components-CIS}.
Suppose next that $G\cong L(K_{n,n})$
for some $n\ge 1$ or $G\cong L(G'\circ K_1)$ for some triangle-free graph $G'$. In this case, Theorem \ref{thm:randomly-internally-matchable} implies that $G\cong L(H)$ where $H$ is a randomly internally matchable triangle-free graph without isolated vertices. Hence, by Corollary~\ref{cor:line-triangle-free-CIS-ii}, $G$ is CIS.

\medskip
\noindent{Necessity} (the ``only if'' direction). Let $G$ be connected true-twin-free
 CIS claw-free graph. Suppose first that $G$ is not co-connected and let $G_1,\ldots, G_k$
   (with $k \ge 2$) be the co-components of $G$. Since $k\ge 2$ and $G$ is claw-free,
   each co-component $G_i$ has stability number at most two. Since by
    Lemma \ref{lem:claw-free-CIS-gem-free} $G$ is gem-free, each co-component
    $G_i$ is $P_4$-free. By the recursive structure of $P_4$-free
    graphs~\cite{MR619603}, each $G_i$ is either $K_1$ or is disconnected. Since $\alpha(G_i) \le 2$, we infer each $G_i$ is either $K_1$ or the disjoint union of two complete graphs. Moreover, since $G$ is true-twin-free,
 at most one $G_i$ is the single-vertex graph, and each $G_i$ with
  $\alpha(G_i) = 2$ is isomorphic to $2K_1$. Consequently,
   $G\cong \overline{pK_2+qK_1}$ for some $p\ge 0$ and $q\in \{0,1\}$.

Suppose now that $G$ is co-connected. By Corollary~\ref{cor:claw-free-CIS-line}
 there exists a triangle-free multigraph $H$ such that $G = L(H)$.
Clearly, we may assume that $H$ has no isolated vertices.
Since $G$ is true-twin-free, $H$ is a simple graph. By Corollary~\ref{cor:line-triangle-free-CIS-ii}, $H$ is randomly internally matchable. Theorem \ref{thm:randomly-internally-matchable} implies that $H\cong K_{n,n}$ for some $n\ge 1$, or $H$ is a leaf extension of some (triangle-free) graph $H'$. In the former case, $G \cong L(K_{n,n})$ for some $n\ge 1$. In the latter case, the fact that $G$ is true-twin-free implies that $H$ is the corona of $H'$, that is, $G\cong L(H'\circ K_1)$.
\end{proof}

Theorem~\ref{thm:claw-free-CIS} has the following structural and algorithmic consequences.

\begin{corollary}\label{cor:claw-free-CIS}
A graph $G$ is claw-free and CIS if and only if the true-twin reduction of each component of $G$ is of the form $\overline{pK_2+qK_1}$ for some $p\ge 0$ and $q\in \{0,1\}$, $L(K_{n,n})$ for some $n\ge 1$, or $L(G'\circ K_1)$ for some triangle-free graph $G'$.
\end{corollary}

\begin{proof}
Immediate from Lemmas~\ref{lem:components-CIS} and~\ref{lem:components-claw-free}, Theorem~\ref{thm:claw-free-CIS}, and the fact that all graphs of the form $\overline{pK_2+qK_1}$, $L(K_{n,n})$, or $L(G'\circ K_1)$ are claw-free.
\end{proof}

\begin{corollary}\label{cor:recognition}
There is a polynomial-time algorithm for recognizing claw-free CIS graphs.
\end{corollary}

\begin{proof}
Let $\mathcal{C}$ be the class of claw-free CIS graphs and let $G$ be a graph that we want to test for membership in $\mathcal{C}$. Since the components of $G$ can be computed in linear time and $G\in \mathcal{C}$ if and only if each component of $G$ is in $\mathcal{C}$, we may assume that $G$ is connected.
Furthermore, since the true-twin reduction of $\mathcal{C}$ can be computed in polynomial time
and $G\in \mathcal{C}$ if and only it its true-twin reduction is in $\mathcal{C}$, we may assume that $G$ is true-twin-free.
By Corollary~\ref{cor:claw-free-CIS}, it suffices to verify whether
(i) $G\cong \overline{pK_2+qK_1}$ for some $p\ge 0$ and $q\in \{0,1\}$, (ii) $G\cong L(K_{n,n})$ for some $n\ge 1$, or (iii) $G\cong L(G'\circ K_1)$ for some triangle-free graph $G'$. Determining whether (i) holds can be done in linear time simply by computing the vertex degrees, since an $n$-vertex graph $G$ is isomorphic to $\overline{pK_2+qK_1}$ with $n = 2p+q$ if and only if $G$ has $q$ vertices of degree $|V(G)|-1$ and $n-q$ vertices of degree $|V(G)|-2$. To determine whether at least one of conditions (ii) or (iii) holds (and if so, which one), we first compute a graph $H$ such that $G= L(H)$ or determine that there is no such graph in linear time~\cite{MR0347690,MR0424435}. If such a graph $H$ exists, it is unique if $G$ is connected graph of order at least four (which we can assume)~\cite{MR1506881}. Testing if $H \cong K_{n,n}$ in polynomial time is straightforward.
Finally, it remains to test if $H\cong H'\circ K_1$ for some triangle-free graph $H'$. A graph $H'$ such that $H\cong H'\circ K_1$, if there is one, can be computed in linear time by identifying the leaves in $H$ and their neighbors. Since $G$ is connected and of order at least four, $H$ is connected and with at least four edges. Therefore, graph $H'$, if it exists, is unique. It remains to test if $H'$ is triangle-free, which, clearly, is doable in polynomial time.
\end{proof}

\section{Bounding the order of CIS graphs by the product of stability and clique numbers}\label{sec:bounds}

\subsection{Answering a question of Dobson et.~al}

We answer the following question of Dobson et al.~\cite{MR3278773} in the negative.

\begin{question}\label{q:CIS-1}
Does every CIS graph $G$ satisfy $|V(G)|\le \alpha(G)\cdot \omega(G)$?
\end{question}

In fact, as we show in the next theorem, the order of CIS graphs cannot even be bounded from above by any linear function of the product $\alpha(G)\omega(G)$.

\begin{theorem}\label{thm:counterexamples}
For every positive integer $k$ there exists a CIS graph $G_k$ such that
$ |V(G_k)|>k\cdot \alpha(G_k)\cdot \omega(G_k)$.
\end{theorem}

\begin{proof}
Kim proved in~\cite{MR1369063} that there exists a positive integer $n_0$ such that for all $n\ge n_0$, there exists an \hbox{$n$-vertex} triangle-free graph $H_n$ such that $\alpha(H_n) \le 9\sqrt{n\log n}$. We may assume that $H_n$ has no isolated vertices, since otherwise we can add, as long as necessary, for each isolated vertex $v$, an edge joining $v$ with some other vertex. Note that modifying $H_n$ this way does not create any triangles, it does not change the number of vertices, and it does not increase the stability number.

For a positive integer $k$, let $n_k$ be the smallest positive integer such that $n_k\ge n_0$ and $n_k\ge 54k\sqrt{n_k\log{n_k}}$.
Let $G_k$ be the graph obtained from $H_{n_k}$ by the following two-step procedure:
\begin{enumerate}
  \item First, construct a graph $H'_k$ by gluing a triangle along each edge of $H_{n_k}$. Formally, $V(H'_k) = V(H_{n_k})\cup\{v^e: e\in E(H_{n_k})\}$ and $E(H'_k) = E(H_{n_k})\cup\{uv^e: u$ is an endpoint of $e$ in $H_{n_k}\}$.
  \item Second, let $p = 6kn_k$ and for each vertex $v\in V(H_{n_k})$ in the graph $H'_k$, substitute $pK_p$ (the disjoint union of $p$ copies of $K_p$) for $v$. Call the resulting graph $G_k$.
\end{enumerate}

Since $H_{n_k}$ is triangle-free and without isolated vertices, its maximal cliques are its edges. It follows that the maximal cliques of $H'_k$ are the triangles consisting of the two endpoints of an edge $e$ of $H_{n_k}$ together with the new vertex $v^e$ associated with that edge.
In particular, every maximal clique of $H'_k$ is simplicial, and hence the graph $H'_k$ is CIS.

Clearly, $pK_p$ is a CIS graph. Since the class of CIS graphs is closed under substitution~\cite{AndBorGur}, we infer that the graph $G_k$ is also CIS. Its clique and stability numbers can be estimated as follows:
\begin{itemize}
  \item $\omega(G_k) =2p+1\le 3p$.

    A clique in $G_k$ of size $2p+1$ can be obtained by choosing any edge $e$ of $H_{n_k}$, taking two cliques of size $p$, one from each copy of $pK_p$ replacing an endpoint of $e$ in $G_k$, and vertex $v^e$. It is not difficult to see that there are no larger cliques in $G_k$.

  \item $\alpha(G_k) < 9p\sqrt{n_k\log  n_k} + n_k^2$.

  The stability number of the graph obtained from a graph $F$ by substituting a graph $H_v$ into every vertex $v$ of $F$ equals the maximum total weight of a stable set in the graph $F$ in which each vertex $v\in V(F)$ has weight equal to the stability number of $H_v$ (see, e.g.,~\cite{MR2463423}). Therefore, the stability number of $G_k$ equals the maximum total weight of a stable set in the graph $H'_k$ in which each vertex $v\in V(H_{n_k})$ has weight $p$ and all other vertices have unit weight. Let $S$ be a corresponding maximum-weight stable set of $H'_k$. Writing $S = S_p\cup S_1$ where $S_p = S\cap V(H_{n_k})$ and $S_1 = S\setminus V(H_{n_k})$, we see that the total weight of $S_p$ is at most $p\cdot \alpha(H_{n_k})$, while the total weight of $S_1$ is at most $|V(H'_k)\setminus V(H_{n_k})|= |E(H_{n_k})|$. It follows that
  $\alpha(G_k)\le p\cdot \alpha(H_{n_k})+|E(H_{n_k})|<9p \sqrt{n_k\log  n_k} + n_k^2\,,$ as claimed.
  \end{itemize}
Consequently, we have
$k\cdot\alpha(G_k)\cdot\omega(G_k) < k\cdot(9p\cdot \sqrt{n_k\log  n_k} + n_k^2)\cdot 3p = p\cdot \big(27k\sqrt{n_k\log  n_k}\cdot p+3kn_k^2\big)
\le p\cdot \left(pn_k/2+pn_k/2\right)= p^2n_k< |V(G_k)|$,
where the second inequality follows from
$n_k\ge 54k\sqrt{n_k\log{n_k}}$ and $p = 6kn_k$, and the last one from
$|V(G_k)| = p^2n_k + |E(H_{n_k})|$.
This completes the proof.
\end{proof}

\subsection{A positive answer for claw-free graphs}

In order to show that Question~\ref{q:CIS-1} has a positive answer for claw-free CIS graphs, we first show a property of weighted randomly internally matchable graphs, which follows easily from the characterization of randomly internally matchable graphs given by Theorem~\ref{thm:randomly-internally-matchable}.
For this we need some definitions.

Given a weighted graph $(H,w)$ and a set $X\subseteq E(H)$, we denote by $w(X)$ the total weight of edges in $X$, that is, $w(X) = \sum_{e\in X}w(e)$. Given a vertex $v\in V$, we denote by $E(v)$ the set of all edges having $v$ as endpoint, and define its \emph{weighted degree} as $d_w(v) = \sum_{e\in E(v)}w(e)$. We denote by $\Delta_w(H)$ the maximum weighted degree of a vertex in $H$. The maximum size of a matching in a graph $H$ is its \emph{matching number}, denoted by $\nu(H)$. A \emph{maximum matching} in $H$ is a matching of a size $\nu(H)$.

\begin{lemma}\label{lem:randomly-internally-matchable-weights}
Let $(H,w)$ be a weighted graph such that $H$ is a connected randomly internally matchable graph. Then $w(E(H))\le \Delta_w(H)\cdot \nu(H)\,.$
\end{lemma}

\begin{proof}
By Theorem~\ref{thm:randomly-internally-matchable}, $H$ satisfies one of the following conditions: (i) $H\cong K_{2n}$ for some $n\ge 1$, (ii) $H\cong K_{n,n}$ for some $n\ge 1$, or (iii) $H$ is a leaf extension of some $n$-vertex graph $H'$. It is not difficult to see that in either case, the matching number of $H$ equals $n$. Hence, we want to show the inequality $w(E(H))\le n\cdot \Delta_w(H)$. If $H\cong K_{2n}$, then $2w(E(H))= \sum_{x\in V(H)}w(E(x)) \le 2n\cdot \Delta_w(H)$. If $H\cong K_{n,n}$, with a bipartition $\{X,Y\}$, then $w(E(H)) = \sum_{x\in X}w(E(x))\le n\cdot \Delta_w(H)$. Finally, if $H$ is a leaf extension of some $n$-vertex graph  $H'$, then $w(E(H)) \le \sum_{x\in V(H')}w(E(x))\le n\cdot \Delta_w(H)$.
Thus, in either case, the desired inequality holds.
\end{proof}

From Lemma~\ref{lem:randomly-internally-matchable-weights} we derive one more intermediate lemma.

\begin{lemma}\label{lem:line-graphs-of-multigraphs}
Let $H$ be a connected triangle-free multigraph and let
$G = L(H)$. If $G$ is CIS, then $|V(G)|\le \alpha(G)\cdot \omega(G)$.
\end{lemma}

\begin{proof}
Let $(H,w)$ be a weighted graph such that $H$ is a connected triangle-free graph, and assume that $G = L(H,w)$ is CIS (cf.~Section~\ref{subsec:prelim-line}).
Let $G' = L(H)$ be the usual line graph of $H$. It is not difficult to see that $G$ can be obtained from $G'$ by substituting, for each vertex $v \in V(G')$, a clique of size $w(e_v)$ where $e_v$ is the edge of $H$ corresponding to $v$. Since $G$ is CIS if and only if $G'$ is CIS (see, e.g.,~\cite{MR3141630}), we infer that $G'$ is CIS. By Corollary~\ref{cor:line-triangle-free-CIS-ii}, $H$ is randomly internally matchable. Since $H$ is connected, Lemma~\ref{lem:randomly-internally-matchable-weights} implies that
$w(E(H))\le \Delta_w(H)\cdot \nu(H)$.

 Clearly, the stability number of $G$ is given as $\alpha(G) = \alpha(L(H,w)) = \alpha(L(H)) = \nu(H)$.
Moreover, the clique number of $G$ is the maximum weight of a clique in $L(H)$ with respect to the weight function assigning to each vertex $v\in L(H)$ weight $w(e_v)$, where $e_v$ is the edge of $H$ corresponding to $v$. Since $H$ is triangle-free, every clique in $L(H)$ corresponds to a set of edges in $H$ with a fixed common endpoint. It follows that the clique number of $G$ equals the maximum weighted degree $\Delta_w(H)$. Since also $|V(G)| = w(E(H))$, the inequality $w(E(H))\le \Delta_w(H)\cdot \nu(H)$ implies that $|V(G)|\le \alpha(G)\cdot \omega(G)$, as claimed.
\end{proof}

\begin{theorem}\label{thm:claw-free-CIS-alpha-omega}
If $G$ is a CIS claw-free graph, then $|V(G)|\le \alpha(G)\cdot \omega(G)$.
\end{theorem}

\begin{proof}
Suppose for a contradiction that the theorem fails and let $G$ be a counterexample with the smallest possible number of vertices.
That is, $G$ is a claw-free CIS graph such that $|V(G)|>\alpha(G)\cdot \omega(G)$, but every smaller claw-free CIS graph $G'$ satisfies
$|V(G')|\le \alpha(G')\cdot \omega(G')$.

First we show that $G$ is connected. Suppose not, and let $G$ be the disjoint union of two  graphs $G_1$
and $G_2$. Then each of $G_1$ and $G_2$ is a smaller claw-free CIS graph,
 therefore \hbox{$|V(G_i)|\le \alpha(G_i)\cdot \omega(G_i)$} holds for $i\in \{1,2\}$ by the minimality of $G$.
We have $\omega(G) = \max\{\omega(G_1),\omega(G_2)\}$.
Without loss of generality we may assume that $\omega(G) = \omega(G_1)\ge \omega(G_2)$.
We thus have
$|V(G)|= |V(G_1)|+ |V(G_2)| \le  \alpha(G_1)\cdot \omega(G_1) + \alpha(G_2)\cdot \omega(G_2)
\le  (\alpha(G_1)+\alpha(G_2))\cdot \omega(G_1)
= \alpha(G)\cdot \omega(G)$.
Hence, $G$ is not a counterexample. This contradiction shows that $G$ is connected.

Next, we show that $G$ is co-connected. If not, then
 $\overline{G}$ is the disjoint union of two  smaller graphs $H_1$ and $H_2$.
  Since  $\overline{G}$ is CIS, we have that both $H_1$ and $H_2$ are CIS.
 Therefore $\overline{H_1}$ and $\overline{H_2}$  are CIS  claw-free graphs of smaller order than $G$, which implies
  $|V(\overline{H_i})|\le \alpha(\overline{H_i})\cdot \omega(\overline{H_i})$ for  $i\in \{1,2\}$.
Similar arguments as above show that $|V(G)|= |V(\overline{H_1})|+ |V(\overline{H_2})| \leq
\alpha(\overline{G})\cdot \omega(\overline{G})=\alpha(G)\cdot \omega(G)$, a contradiction. Thus, $G$ is co-connected.

Now, since $G$ is a connected and co-connected claw-free CIS graph, Corollary~\ref{cor:claw-free-CIS-line} implies that $G$ is the line graph of a connected triangle-free multigraph. By Lemma~\ref{lem:line-graphs-of-multigraphs}, we have
$|V(G)|\le \alpha(G)\cdot \omega(G)$, implying that $G$ is not a counterexample. This contradiction completes the proof.
\end{proof}

\section{An open question}\label{sec:EH}

Question~\ref{q:CIS-1} has been answered in the negative. The following relaxation of it is still open.

\begin{question}\label{q:CIS}
Is there an integer $k$ such that every
CIS graph $G$ satisfies $|V(G)|\le (\alpha(G)\cdot \omega(G))^k$?
\end{question}

A graph class $\mathcal{G}$ is said to satisfy the \emph{Erd\H{o}s-Hajnal property} if there exists some $\varepsilon>0$ such that $\max\{\alpha(G), \omega(G)\}\ge |V(G)|^\varepsilon$ holds for all graphs $G\in \mathcal{G}$. The well-known Erd\H{o}s-Hajnal Conjecture~\cite{MR1031262} asks whether for every graph $F$, the class of $F$-free graphs has the Erd\H{o}s-Hajnal property. The conjecture is still open (see~\cite{MR3150572} for a survey).
It is not difficult to see that, using this terminology, Question~\ref{q:CIS} can be equivalently phrased as follows.

\begin{question}\label{q:EH}
Does the class of CIS graphs have the Erd\H{o}s-Hajnal property?
\end{question}

Note, however, that it is possible that Question~\ref{q:EH} bears only superficial resemblance with the Erd\H{o}s-Hajnal Conjecture. This is because every graph is an induced subgraph of a CIS graph (see, e.g.,~\cite{AndBorGur}); furthermore, for a graph $F$, the class of $F$-free graphs is a subclass of the class of CIS graphs if and only if $F$ is an induced subgraph of $P_4$. It is thus in principle possible that
the Erd\H{o}s-Hajnal Conjecture is true, while Question~\ref{q:EH} has a negative answer, or vice versa.

\subsection*{Acknowledgement}

The authors are grateful to Nevena Piva\v{c} for helpful discussions and remarks. The work of the third author is supported in part by the Slovenian Research Agency (I0-0035, research program P1-0285, and research projects N1-0032, J1-7051, and J1-9110). Part of the work for this paper was done
in the framework of a bilateral project between Argentina and Slovenia, partially financed by the Slovenian Research Agency (BI-AR/$15$--$17$--$009$).

\end{document}